\documentclass[letterpaper,12pt,reqno]{amsart}
\usepackage[margin=1in]{geometry}

\usepackage[parfill]{parskip}
\usepackage[foot]{amsaddr}
\usepackage{amssymb}
\usepackage{amsmath}
\usepackage{amsthm}
\usepackage{mathtools}
\usepackage{latexsym}
\usepackage{amsbsy}
\usepackage{bbm}
\usepackage{amsfonts}
\usepackage{graphicx}
\usepackage{caption}
\usepackage{enumerate}
\usepackage{enumitem}
\usepackage{mathrsfs}
\usepackage{dsfont}
\usepackage{url}
\usepackage{color}
\usepackage{algorithm}% http://ctan.org/pkg/algorithm
\usepackage{algpseudocode}% http://ctan.org/pkg/algorithmicx
\usepackage[citecolor=blue, urlcolor=blue, linkcolor=red, colorlinks=true, bookmarksopen=true]{hyperref}
\allowdisplaybreaks

\mathtoolsset{showonlyrefs}

\DeclareMathOperator\E{\mathbb{E}}
\DeclareMathOperator\Prob{\mathbb{P}}

\newtheorem{theorem}{Theorem}[section]
\newtheorem{lemma}[theorem]{Lemma}
\newtheorem{proposition}[theorem]{Proposition}
\newtheorem{corollary}[theorem]{Corollary}
\newtheorem{definition}[theorem]{Definition}
\newtheorem{remark}[theorem]{Remark}

\numberwithin{equation}{section}
\makeatother

\title{A sharp leading order asymptotic of the diameter of a long range percolation graph }
\author[Tianqi Wu]{{Tianqi} Wu}
\address{Department of Mathematics, UCLA}
\email{timwu@math.ucla.edu}
\date{\today} 

\begin{document}
	
	\maketitle
	
	\begin{abstract}
	Many real-world networks exhibit the so-called \emph{small-world phenomenon}: their typical distances are much smaller than their sizes. One mathematical model for this phenomenon is a \emph{long-range percolation} graph on a $d$-dimensional box $\{0, 1, \cdots, N\}^d$, in which edges are independently added between far-away sites with probability falling off as a power of the Euclidean distance. A natural question is how the resulting diameter of the box of size $N$, measured in \emph{graph-theoretical distance}, scales with $N$. This question has been intensely studied in the past and the answer depends on the exponent $s$ in the connection probabilities. In this work we focus on the critical regime $s = d$ studied earlier in a work by Coppersmith, Gamarnik, and Sviridenko and improve the bounds obtained there to a sharp leading-order asymptotic, by exploiting the high degree of concentration due to the large amount of independence in the model. 
	\end{abstract} 

\section{Background and Main Result}

Many real-world networks exhibit the \emph{small-world} phenomenon: their typical distances are much smaller than their sizes. The term originates from the old paper \cite{Milgram67} which suggested that two average Americans are just six acquaintances away from each other (the so-called ``six degrees of separation''), and the phenomenon has received renewed interest since advances in communication and transportation technologies, in particular the Internet, have tremendously increased global connectivity in recent decades. 

A natural way to model this phenomenon is a \emph{long-range percolation graph} on the hypercubic lattice $\mathbb{Z}^d$, in which nearest neighbors are connected \emph{a priori} and, independently at random, a new edge is added between each pair of non-adjacent sites $x, y$ with probability $p_{x, y} := 1 - \exp{\left(-\beta |x-y|^{-s}\right)} \approx \beta |x-y|^{-s}$, for some parameter $\beta > 0, s > 0$. The objects of interest are how the \emph{typical graph-theoretical distance} $D(x, y)$ scales with the Euclidean distance $|x-y|$ and, when restricted to a finite box $[N]^d := \{0, 1, \cdots, N\}^d$, how the \emph{diameter} $D_N$ in the graph-theoretical distance scales with the box size $N$. The motivation is that adding long edges, even quite sparsely, could substantially shorten the typical distance of the network (see e.g.~the short article \cite{WS98}).  

This type of models was introduced and studied by \cite{BB01} for a one-dimensional case and extended to a multi-dimensional version by \cite{CGS02}. Five regimes of behavior have been identified based on the exponent $s$: $s < d$, $s = d$, $d < s < 2d$, $s = 2d$, and $s > 2d$. 
\begin{itemize}
	\item In the regimes $s < d$, \cite{BKPS01} showed the graph diameter $D_N$ approaches the (deterministic) number $\lceil d/(d-s) \rceil$ as $N\rightarrow \infty$ using the \emph{stochastic dimension} method. 
	\item In the regime $d < s < 2d$, Biskup (\cite{Biskup04}, \cite{Biskup11}) showed a poly-logarithmic scaling for both the graph distance and the graph diameter: $D(x, y) = (\log |x-y|)^{\Delta + o(1)}$ as $|x-y| \rightarrow \infty$ and $D_N = (\log N)^{\Delta + o(1)}$ as $N\rightarrow \infty$, where $\Delta^{-1}:= \log_2 (2d/s)$. For the graph distance, this result was later improved to $D(x, y) = \Theta(1) (\log |x-y|)^{\Delta}$ for a continuum model in \cite{BL19} and, more recently, for the model on $\mathbb{Z}^d$ in \cite{BK22}. The main idea is identifying a ``binary hierarchy'' of edges forming the path between two given sites. 
	\item In the regimes $s > 2d$, \cite{Berger04} showed the graph distance resumes linear scaling with the Euclidean distance. 
	\item The critical regime $s = 2d$, where the model is scale-invariant, is still largely open. In dimension $d = 1$, \cite{DS13} showed that there exists a \emph{distance exponent} $\theta(\beta) \in (0, 1)$ such that the graph distance $D(0, N) = \Theta(1) N^{\theta(\beta)}$. Recently, analogous results were obtained and extended for both the graph distance and the graph diameter \emph{in all dimensions} $d$ by B{\"a}umler (\cite{Baumler22a}, \cite{Baumler22b}).
\end{itemize}

Our work focuses on the critical regime $s = d$. Precisely, let $$[N]^d \equiv \{0, 1, \cdots, N\}^d$$ denote a hypercubic lattice of length $N$, and let $\|x\|$ denote the $\ell^1$-norm of $x \in \mathbb{Z}^d$. We consider a random graph $G(N)$ on the hypercube $[N]^d$ in which every pair of sites $x, y \in [N]^d$ is independently connected with probability 
\begin{equation}\label{eq:connection_probabilities}
\begin{cases}
	1 & \text{if} \ \|x-y\| = 1, \\
	1 - \exp{\left(-\frac{\beta}{\|x-y\|^{d}}\right)} & \text{otherwise},
\end{cases}
\end{equation}
where $\beta > 0$ is some fixed parameter. Let $D(x, y)$ denote the graph-theoretical distance between $x, y$ in $G(N)$ and let $D_N$ denote the associated diameter of $G(N)$. The authors of \cite{CGS02} proved the following result:

\begin{proposition}\label{prop_1} There exist constants $C_1, C_2$ (which may depend on $\beta, d$) such that
	\[\lim_{N\rightarrow \infty} \Prob \left(\frac{C_1 \log N}{\log \log N} \leq D_N \leq \frac{C_2 \log N}{\log \log N} \right) = 1. \]
\end{proposition}

Let us denote by $B_m(x)$ the \emph{ball} of radius $m$ centered at $x$ in the \emph{graph distance}, i.e.
\[
B_m(x) := \{y\in [N]^d: D(x, y) \leq m\}.
\] 
Proposition \ref{prop_1} then may be formulated in the following equivalent way: 
\[
	\lim_{N\rightarrow \infty} \mathbb{P}(B_m(x) = [N]^d) =	\begin{cases}
 1 & m \geq \frac{C_2 \log N}{\log \log N} \\
 0 & m\leq \frac{C_1 \log N}{\log \log N}
	\end{cases}
	\]
for some constants $C_1, C_2$ (which may depend on $\beta, d$). 

Here is a heuristic explanation of this result. The critical exponent $s = d$ in the connection probabilities \eqref{eq:connection_probabilities} means that a typical site $x$ in $G(N)$ has $\Theta(\log N)$ neighbors with high probability. This suggests that as $m$ increments, the ball $B_m(x)$ grows like a ``tree'' with branching degree $\Theta(\log N)$, if we ignore for the moment there may be multiple connections to the same site. Therefore, starting from $x$, in $m$ steps we should be able to reach $\left(\Theta(\log N)\right)^m$ sites. At $m = d \log N/\log \log N$, this amounts to $\Theta(N^d)$ sites. This suggests the constants $C_1, C_2$ can be brought arbitrarily close to $d$. The arguments in \cite{CGS02} already show that $C_1 = d - \varepsilon$ for any $\varepsilon > 0$, but the value of $C_2$ is much bigger than $d$. In this work we show that we can also take $C_2 = d + \varepsilon$ for any $\varepsilon > 0$.

\begin{theorem}\label{main_theorem} For any $\varepsilon > 0$,
	\[\lim_{N\rightarrow \infty} \Prob \left(D_N \leq  \frac{(1+\varepsilon) d \log N}{\log \log N}\right) = 1. \]
	In particular, $D_N \frac{\log \log N}{\log N} \rightarrow d$ as $N\rightarrow \infty$ in probability. 
\end{theorem}

The next step is to go beyond the leading order, which requires a deeper look at the spatial distribution of the ``tree'' and see how ``homogeneously'' the ``tree'' is distributed in the box. This is a question currently under investigation.

\section{Proof of Theorem \ref{main_theorem}}

The main idea of the proof is quite simple: we find some natural numbers $m_1, m_2$ such that for any two sites $x, y$, the ball $B_{m_1}(x)$ is connected to the ball $B_{m_2}(y)$ by an edge with high probability. The sum $m_1 + m_2$ then provides an upper bound on the graph distance between $x$ and $y$ in $G(N)$ and we show it can be chosen to be at most $(d+\varepsilon) \log N / \log \log N$. 

To do this, we will use a counting argument for the sizes of the sets $B_{m_1}(x)$, $B_{m_2}(y)$. The basic principle is that for any two disjoint sets $A_1, A_2 \subseteq [N]^d$, there are $|A_1||A_2|$ many ways for them to be connected by an edge, in correspondence to the pairs $(x_1, x_2) \in A_1 \times A_2$. If the sites in $A_1$ are separated by distance $R$ from the sites in $A_2$, then the connection probability for each pair $(x_1, x_2) \in A_1 \times A_2$ is about $1-\exp\left(-\Theta(R^{-d})\right)$. Under the assumption of independence, the total connection probability between $A_1$ and $A_2$ is then about $1 - \exp\left(-\Theta\left(\frac{|A_1||A_2|}{R^d}\right)\right)$. Therefore, the two sets will almost certainly be connected if $|A_1||A_2| \gg R^d$ and will almost certainly fail to be connected if $|A_1||A_2| \ll R^d$. Since the distance between any two sites in $[N]^d$ is at most $\Theta(N)$, this leads us to show that
\[
|B_{m_1}(x)| \cdot |B_{m_2}(y)| \geq \Theta(N^{d+\varepsilon}).
\]
Recall our ``tree'' heuristic from the previous section that says $|B_{m}(x)| = (\Theta(\log N))^m$. Taking logarithms on both sides of the inequality above, this suggests that we take $m_1 = m_2 = \alpha d \log N / \log \log N$, for some $\alpha \in (1/2, 1)$, and show our heuristic lower bound
\[|B_m(x)| \geq [\Theta(\log N)]^m
\]
holds with high probability. Moreover, our ``tree'' heuristic also suggests that we establish this by showing that the \emph{iterative} bound
\[
|B_{m+1} (x)| \geq \Theta(\log N)\cdot |B_m(x)|
\]
holds with high probability. 

Let us now proceed with the rigorous proof. In the following, we write
\begin{itemize}
	\item $x \leftrightarrow y$ to mean site $x$ is connected to site $y$ by an edge;
	\item $x \leftrightarrow U$ to mean site $x$ is connected to \emph{some} site belonging to the set $U$ by an edge.
	\item $U \leftrightarrow V$ to mean \emph{some} site belonging to the set $U$ is connected to \emph{some} site belonging to the set $V$ by an edge.
\end{itemize}
Recall that we defined $B_m(x)$ to be the ball of radius $m$ centered at $x$ in the \emph{graph distance} $D$. We shall need a generalization of this notion. 

\begin{definition} Given a set $U \subseteq [N]^d$, define the \emph{$m$-step expansion} of $U$ to be 
	\[
	B_m(U) := \{y\in [N]^d: D(x, y)\leq m \mbox{ for some } x\in U\}
	\]
and its \emph{boundary} to be the set
\[
\partial B_m(U) := B_m(U) \setminus B_{m-1}(U).
\]
with the convention that $B_{-1}(U) = \emptyset$. The pair 
\[
\mathbf{B}_m(U) := (B_{m-1}(U), \partial B_m(U))
\]
defines a stochastic process indexed by $m = 0, 1, 2, \cdots $ on the state space
\[
\mathbf{S}:= \{(V_1, V_2): V_1, V_2 \text{ disjoint subsets of } [N]^d\}.
\]
It is easy to check that the process $\mathbf{B}_m(U)$ is a \emph{Markov chain} on $\mathbf{S}$ started at $\mathbf{B}_0(U) = (\emptyset, U)$. Indeed, let $G_1, G_2, \cdots$ be i.~i.~d.~realizations of the random graph $G(N)$, then the process $\mathbf{B}_m(U)$ may be iteratively constructed as
\begin{equation}\label{iterative_construction}
	\left\{
	\begin{aligned}
		B_{m}(U) &= B_{m-1}(U) \cup \partial B_m(U), \\
		\partial B_{m+1}(U) &= \{y\in B_m(U)^c: y \leftrightarrow \partial B_m(U) \text{ in } G_{m+1}\}.
	\end{aligned}
	\right.
\end{equation}
By abuse of notation, we shall simply write $\mathbf{B}_m = ( B_{m-1}, \partial B_m)$ for the chain $\mathbf{B}_m(U)$ with $U = B_0$.
\end{definition}

The starting point of our proof of Theorem \ref{main_theorem} is the following concentration inequality for the sizes of the successive boundaries $\partial B_m$.

\begin{lemma}[Concentration of $|\partial B_m|$]\label{chernoff} For $0 < \delta < 1$, $m\geq 0$,
	\[
	\Prob{\Bigl( 1-\delta \leq  \frac{|\partial B_{m+1}|}{\E \bigl[|\partial B_{m+1}(U)| \bigm| \mathbf{B}_m\bigr]} \leq 1+\delta \Bigm\vert \mathbf{B}_m \Bigr)} \geq 1 - 2 \exp{\Bigl(-\frac{\delta^2}{3} \mathbb{E} \bigl[ |\partial B_{m+1}| \bigm| \mathbf{B}_m\bigr]\Bigr)}.
	\]
\end{lemma}

\begin{proof} The iterative construction \eqref{iterative_construction} implies that conditioned on $\mathbf{B}_m$,
	\begin{equation}\label{iterative_construction_2}
	|\partial B_{m+1}| = \sum_{y\in B_m^c} \mathbbm{1}_{\left\{y \leftrightarrow \partial B_m \mbox{ in } G_m \right\}}
	\end{equation}
	is a sum of independent (but not identically distributed) Bernoulli random variables. The inequality then directly follows from a relative error version of the \emph{Chernoff bound for Poisson trials} (Corollary 4.6, \cite{MU05}). 
\end{proof}
\begin{remark} By the results below, the concentration in Lemma \ref{chernoff} gets \emph{exponentially better} with the increment of step number $m$, which means the main source of fluctuations is in the \emph{very first step}, $|\partial B_1|$.
\end{remark}

To use this \emph{iterative} concentration result, we need to estimate $\mathbb{E} \bigl[ |\partial B_{m+1}| \bigm| \mathbf{B}_m\bigr]$ in terms of $|\partial B_m|$. This is provided by the following \emph{iterative} bounds, which form the key ingredient of our proof of Theorem \ref{main_theorem}. 

\begin{lemma}\label{prob_iterate} Given $\alpha \in (0, 1)$, there exist positive, finite constants $c = c(\beta, d, \alpha), C = C(\beta, d), N_* = N_*(\beta, d)$ such that for $N\geq N_*$,
	\begin{equation}\label{eq:prob_iterate} 
		c \log N  \leq \frac{\mathbb{E} \bigl[ |\partial B_{m+1}| \bigm| \mathbf{B}_m\bigr]}{|\partial B_m|} \leq  C \log N  
	\end{equation}
	as long as 
	\begin{equation}\label{size_limit}
	|B_m| \leq N^{\alpha d}.
	\end{equation}
\end{lemma}
\begin{remark} The lower bound in \eqref{eq:prob_iterate} deteriorates as $\alpha$ increases (cf. \eqref{bad_bound} below). To go beyond the leading order asymptotic, we need a better lower bound. 
\end{remark}

We defer the proof of Lemma \ref{prob_iterate} to Section \ref{proof_key_lemma}. Combining Lemma \ref{chernoff} and Lemma \ref{prob_iterate} immediately yields the following explicit iterative bound.
\begin{corollary}\label{prob_iterate_2} Given $\alpha \in (0, 1)$, there exist positive, finite constants $c_1 = c_1(\beta, d, \alpha), c_2 = c_2(\beta, d, \alpha), C = C(\beta, d), N_* = N_*(\beta, d)$ such that for $N\geq N_*$,
	\[
	\Prob{\Bigl( c_1 \log N \leq  \frac{|\partial B_{m+1}|}{|\partial B_{m}|} \leq C \log N \Bigm\vert \mathbf{B}_m \Bigr)} \geq 1 - 2 \exp{\left(-c_2 \log N \cdot |\partial B_{m}| \right)}.
	\]
as long as $|B_m| \leq N^{\alpha d}$.
\end{corollary}

From now on we will assume $B_0$ is of constant size. Iterating the bound in Corollary \ref{prob_iterate_2} yields the following concentration result.
\begin{corollary}\label{concentration} Assume $|B_0| \leq K$. Given $\alpha \in (0, 1)$, there exist positive, finite constants $c_1 = c_1(\beta, d, \alpha), c_2 = c_2(\beta, d, \alpha), C = C(\beta, d), N_* = N_*(\beta, d, \alpha, K)$ such that for $N\geq N_*$,
	\[
	\Prob{\Bigl( (c_1 \log N)^m \leq  \frac{|\partial B_{m}|}{|B_0|} \leq (C \log N)^m \Bigr)} \geq 1 - 2 \exp{\left(-c_2 \log N \cdot |B_0| \right)}.
	\]
as long as $m\leq \frac{\alpha d \log N}{\log \log N}$. 
\end{corollary}
The details of the verification of Corollary \ref{concentration} are postponed to Section \ref{proof_corollaries}. The next proposition follows directly from the previous corollary and makes its implication more explicit.

\begin{corollary}\label{concentration_2} Assume $|B_0| \leq K$. Given $\alpha \in (0, 1)$, there exist positive, finite constants $c_1 = c_1(\beta, d, \alpha), c_2 = c_2(\beta, d, \alpha), C = C(\beta, d), N_* = N_*(\beta, d, \alpha, K)$ such that for $N\geq N_*$,
	\[
	\Prob{\Bigl( N^{\alpha d \, (1 - c_1/\log \log N)} \leq  \frac{|\partial B_{m}|}{|B_0|} \leq N^{\alpha d \, (1 + C /\log \log N)} \Bigr)} \geq 1 - 2 N^{-c_2 |B_0|}.
	\]
	for $m = \left\lfloor \frac{\alpha d \log N}{\log \log N} \right\rfloor$. 
\end{corollary}

We are now ready to prove Theorem \ref{main_theorem}.

\begin{proof}[Proof of Theorem \ref{main_theorem}] Let $x, y$ be any two sites in $[N]^d$, and let $U_x, U_y$ be the ball of radius $R$ centered at $x, y$ in the \emph{ $\ell^1$-distance}, respectively, i.e
	\[
	U_x = \{z \in [N]^d: \|z-x\| \leq R\}, U_y = \{z \in [N]^d: \|z-y\| \leq R\}
	\]
for some positive constant $R$ to be specified later. Consider the process
\[
\mathbf{Z}_t := (\mathbf{B}_{\lfloor (t+1)/2\rfloor} (U_x), \mathbf{B}_{\lfloor t/2\rfloor} (U_y))
\]
indexed by $t = 0, 1, \cdots$. This corresponds to alternating between taking steps of the processes $\mathbf{B}_m(U_x)$ and $\mathbf{B}_m(U_y)$. Define the \emph{stopping time}
\[
\tau := \min(t\geq 0: B_{\lfloor (t+1)/2\rfloor}(U_x) \cap B_{\lfloor t/2\rfloor}(U_y) \neq \emptyset),
\]
then clearly $D(x, y) \leq 2R + \tau$, so for any $m\geq 0$,
\begin{equation}\label{stopping_time_relation}
	\mathbb{P}(D(x, y) \leq 2R + 2m + 1) \geq \mathbb{P}(\tau \leq 2m+1).
\end{equation} 
Moreover, on the event $\tau > 2m$, $B_m(U_x)$ and $B_m(U_y)$ are disjoint, so
\begin{equation}\label{two_balls_meet}
	\mathbb{P}(\tau = 2m+1|\mathbf{Z}_{2m}) = \mathbb{P}(\partial B_m(U_x) \leftrightarrow \partial B_{m}(U_y) \mbox{ in } G_{m+1}|\mathbf{Z}_{2m}) \ \mbox{ on } \{\tau > 2m\},
\end{equation}	
where $G_{m+1}$ is a realization of the random graph $G(N)$ independent from the process $\mathbf{Z}_t$ up to $t = 2m$ (see \eqref{iterative_construction}). It follows from the definition of the random graph $G(N)$ that
\begin{align}
	\mathbb{P}(\partial B_m(U_x) \leftrightarrow \partial B_{m}(U_y) \mbox{ in } G_{m+1}|\mathbf{Z}_{2m}) &= 1-\exp\left(-\sum_{z\in \partial B_m(U_x)} \sum_{w\in \partial B_{m}(U_y)} \frac{\beta}{\|z-w\|^d}\right) \\
	&\geq 1-\exp\left(-|\partial B_m(U_x)|\,|\partial B_{m}(U_y)| \, \frac{\beta}{(dN)^d}\right) \label{uniform_estimate}
\end{align}
where we used the uniform upper bound $\|z-w\|\leq dN$ for any $z, w\in [N]^d$. 

Now, suppose $0 < \varepsilon < 1/2$ is given. From now on, we fix
\[
m = \left\lfloor  \frac{(1/2+\varepsilon/2) d \log N}{\log \log N} \right\rfloor. 
\]
Let $E_x$ be the event that
\[
|\partial B_m(U_x)| \geq N^{(1/2+\varepsilon/4)d}
\]
and define $E_y$ analogously. Then the computations \eqref{two_balls_meet} and \eqref{uniform_estimate} from the preceding discussion yield
\begin{align}
\mathbb{P}(\tau = 2m+1|\mathbf{Z}_{2m})
	%\geq 1-\exp\left(-N^{(1+\varepsilon) d} \frac{\beta}{|dN|^d}\right)
	 \geq 1 - \exp\left(-\frac{\beta}{d^d}\cdot  N^{\varepsilon d/2}\right) \ \mbox{ on } E_x \cap E_y \cap \{\tau > 2m\}.
\end{align}
On the other hand, 
\[
\mathbb{P}(\tau \leq 2m|\mathbf{Z}_{2m}) = 1 \mbox{ on } E_x \cap E_y \cap \{\tau > 2m\}^c
\]
simply by definition, so altogether we have
\begin{equation}\label{eq:1st_error}
\mathbb{P}(\tau \leq 2m+1|E_x \cap E_y) \geq 1 - \exp\left(-\frac{\beta}{d^d}\cdot  N^{\varepsilon d/2}\right).
\end{equation}
Now, by Corollary \ref{concentration_2} (with $\alpha = 1/2 + \varepsilon/2$), there exists a positive integer $N_1$ such that for all $N\geq N_1$,
\begin{align}\mathbb{P}(E_x \cap E_y) &\geq 1 - \mathbb{P}(E_x^c) - \mathbb{P}(E_y^c) \\
	 &\geq 1 - 2 N^{-c_2 |U_x|} - 2 N^{-c_2 |U_y|} \\
	 &\geq 1 - 4 N^{-c_2 R^d} \label{eq:2nd_error}.
\end{align}
for some constant $c_2 > 0$ independent of $R$. Combining the estimates \eqref{eq:1st_error} and \eqref{eq:2nd_error}, we obtain
\begin{align}
\mathbb{P}(\tau \leq 2m+1) &\geq \mathbb{P}(\tau \leq 2m+1|E_x \cap E_y) \, \mathbb{P}(E_x \cap E_y)  \\
 &\geq 1  - \exp\left(-\frac{\beta}{d^d}\cdot  N^{\varepsilon d/2}\right) - 4 N^{-c_2 R^d}.
\end{align}
By choosing $R$ constantly large, we can guarantee that there exists some positive integer $N_2$ such that for all $N\geq N_2$
\begin{align}\label{stopping_time_estimate}
	 \mathbb{P}(\tau \leq 2m+1) \geq 1 - 4 N^{-2d-1}.
\end{align}
Using estimate \eqref{stopping_time_estimate} with relation \eqref{stopping_time_relation}, union bound over all pairs $x, y$ in $[N]^d$ then gives
\[
\mathbb{P}(D_N \leq 2R+2m+1) \geq 1 - 4 N^{-1},
\]
which readily implies Theorem \ref{main_theorem}.
\end{proof}

\section{Proof of Lemma \ref{prob_iterate}}\label{proof_key_lemma}

We start by taking conditional expectation of both sides of equation \eqref{iterative_construction_2}: 
\[
\mathbb{E} \bigl[ |\partial B_{m+1}| \bigm| \mathbf{B}_m\bigr] = \sum_{y\in B_m^c} \mathbb{P}(y\leftrightarrow \partial B_m \mbox{ in } G_m|\mathbf{B}_m).
\]
To help estimate this sum of probabilities, we introduce the following statistic. 

\begin{definition} [Weight of a site relative to a region] Given a site $y \in [N]^d$ and a subset $S \subset [N]^d$, we define the weight of $y$ relative to $S$ to be
	\[\rho(y, S) := \sum_{x\in S, x\neq y} \frac{\beta}{|x-y|^d}.
	\]
\end{definition} 
It follows from the definition of the random graph $G(N)$ that
\[
\Prob(y\leftrightarrow S \mbox{ in } G(N)) = 1 - e^{-\rho(y, S)}
\]
for $y\notin S$, and thus
\begin{equation}\label{sum_of_prob}
	\mathbb{E} \bigl[ |\partial B_{m+1}| \bigm| \mathbf{B}_m\bigr] = \sum_{y\in B_m^c} 1 - \exp{\left(-\rho(y, \partial B_m)\right)}.
\end{equation}
The terms in this summation will be estimated by the following elementary bounds. 
\begin{lemma}\label{elementary} For any real number $\rho \geq 0$, 
	\[
	\rho (1 - \rho) \leq 1-e^{-\rho} \leq \rho.
	\]
\end{lemma}
\begin{remark} These bounds follow from a standard Taylor expansion. It is only sharp for small $\rho$. In that regime, $1-e^{-\rho}$ behaves like the linear function $\rho$. 
\end{remark}

\begin{lemma}\label{weight_estimate} There exist finite, positive constants $c = c(d), C = C(d), N_*$ such that for $N\geq N_*$, 
	\[
	c \log \left(\frac{N^d}{|V^c|}\right) \leq \rho(x, V) \leq C \log |V|
	\]
	for any site $x \in [N]^d$ and set $V\subseteq [N]^d$.
\end{lemma}
\begin{remark} This estimate is very simple-minded: the upper and lower bounds come from considering the extremal cases where $V$ is clustered around $x$ and $V^c$ is clustered around $x$, respectively, and summing $1/|x-y|^d$ over sites $y\in V, y\neq x$. Consequently, the gap between the lower and upper bound becomes quite wide when the size of $V$ gets small. If we have a better understanding of how $V$ is distributed in the box relative to $x$, we would be able to get a tighter control. 
\end{remark}

We are now ready to prove Lemma \ref{prob_iterate}.

\begin{proof}[Proof of Lemma \ref{prob_iterate}]
	The iterative upper bound in \eqref{eq:prob_iterate} is an easy consequence of the identity \eqref{sum_of_prob}. By the upper bound in Lemma \ref{elementary}, 
	\begin{align}
		\sum_{y\in B_m^c} 1 - \exp{\left(-\rho(y, \partial B_m)\right)} &\leq \sum_{y\in B_m^c} \rho(y, \partial B_m) \\
		&= \sum_{y\in B_m^c} \sum_{x\in \partial B_m} \frac{\beta}{|y-x|^d} \\
		&= \sum_{x\in \partial B_m} \rho(x, B_m^c) \\
		&\leq |\partial B_m| \cdot C\log N
	\end{align}
	where we exploited the additive definition of the weight and used the upper bound from Lemma \ref{weight_estimate}. (Notice that the condition \eqref{size_limit} is not required for the upper bound, which is why the constant $C$ appearing in \eqref{eq:prob_iterate} does not depend on $\alpha$.) 
	
	To obtain the iterative lower bound in \eqref{eq:prob_iterate}, we would like to repeat the calculation above but applying the lower bound from Lemma \ref{elementary} instead. However, this strategy is going to match the upper bound only when $\rho = o(1)$. Therefore, we shall restrict the sum \eqref{sum_of_prob} to be over the sites in
	\[
	V_m := \{y\in B_m^c: \rho(y, \partial B_m) < \omega_m\}
	\]
	for some threshold $\omega_m = o(1)$ to be specified later. Then
	\begin{align}
		\sum_{y\in B_m^c} 1 - \exp{\left(-\rho(y, \partial B_m)\right)} &\geq \sum_{y\in V_m} 1 - \exp{\left(-\rho(y, \partial B_m)\right)} \\
		&\geq \sum_{y\in V_m} \rho(y, \partial B_m) (1-\omega_m) \\
		&= \sum_{z\in \partial B_m} \rho(z, V_m) (1-\omega_m) \\
		&\geq |\partial B_m| \cdot c \log \left(\frac{N^d}{|V_m^c|}\right) (1-\omega_m)
	\end{align}
	where we used the lower bound in Lemma \ref{weight_estimate} in the last step. It remains to choose the set $V_m$ by specifying the threshold $\omega_m$. We want to make sure both $\omega_m = o(1)$ and $\log \left(\frac{N^d}{|V_m^c|}\right) = \Theta(\log N)$. To achieve both, we may simply take
	\[
	\omega_m = (\log N)^{-1}.
	\] 
	Then by the pigeonhole principle,
	\begin{align}
		|V_m^c \setminus B_m| &= |\{y\in B_m^c: \rho(y, \partial B_m) \geq \omega_m\}| \\
		&\leq \frac{\displaystyle\sum_{y\in B_m^c} \rho(y, \partial B_m)}{\omega_m} \\
		&\leq C \log N \cdot |\partial B_m| \cdot \log N
	\end{align}
	by the same estimates we used for the upper bound. Thus,
	\begin{align}
		|V_m^c| = |V_m^c \setminus B_m| + |B_m| \leq C (\log N)^2 |\partial B_m| + |B_m| \leq (1 + C(\log N)^2) N^{\alpha d},
	\end{align}
	where we used the condition \eqref{size_limit}, and so
	\begin{align}\label{bad_bound}
		\log \left(\frac{N^d}{|V_m^c|}\right) &\geq \log \left(\frac{N^{(1-\alpha)d} }{ 1 + C(\log N)^2}\right) = \left(d(1-\alpha) - \Theta \left(\frac{\log \log N}{\log N} \right)\right) \log N. 
	\end{align}
\end{proof}

\section{Proof of Corollary \ref{concentration}}\label{proof_corollaries}

Let $E_m$ be the event that 
	\[c_1 \log N \leq  \frac{|\partial B_{m}|}{|\partial B_{m-1}|} \leq C \log N
	\]
	where $c_1, C$ are the constants appearing in Corollary \ref{prob_iterate_2}. Let
	\[
	\mathcal{E}_m = \bigcap_{i=1}^m E_i,
	\]
	then
	\[(c_1 \log N)^m |B_0| \leq	|\partial  B_{m}| \leq (C \log N)^m |B_0|\]
	on the event $\mathcal{E}_m$. In particular,
	\begin{align}
	|B_m| = \sum_{i=0}^m |\partial B_i| &\leq \frac{(C\log N)^{m+1} - 1}{(C\log N) - 1} |B_0| \\
	&\leq (C\log N)^m K \left(1+\Theta(1/\log N)\right) \\
	&\leq N^{\alpha d (1 + \Theta(1/\log \log N))}
	\end{align}
	as long as $m\leq \frac{\alpha d \log N}{\log \log N}$. We will use the following version of the union bound:
	\begin{align}
		\Prob\left(\mathcal{E}_m^c\right) = \sum_{i=1}^m  \Prob\left( E_i^c \cap \mathcal{E}_{i-1}\right) = \sum_{i=1}^m  \Prob\left( E_i^c| \mathcal{E}_{i-1}\right) \Prob (\mathcal{E}_{i-1}) \leq \sum_{i=1}^m \Prob\left( E_i^c | \mathcal{E}_{i-1}\right).
	\end{align}
	By Corollary \ref{prob_iterate_2} applied to $\tilde \alpha = \alpha + \varepsilon$ for some $\varepsilon > 0$, for sufficiently large $N$,
	\begin{align} 
		\Prob\left( E_i^c | \mathcal{E}_{i-1}\right) \leq 2\exp\left(-\tilde c_2 \log N \cdot |\partial B_m|\right) \leq 2\exp\left(-\tilde c_2 \log N \cdot (c_1 \log N)^{i-1} |\partial B_0|\right).
	\end{align}
where $\tilde c_2$ is the constant $c_2$ appearing in Corollary \ref{prob_iterate_2} for $\tilde \alpha = \alpha+\varepsilon$. Therefore,
	\begin{align}
		\Prob\left(\mathcal{E}_m\right)	= 1 - \Prob\left(\mathcal{E}_m^c\right) &\geq 1 - \sum_{i=1}^m 2\exp\left(-\tilde c_2 \log N \cdot (c_1 \log N)^{i-1} |B_0|\right) \\
		& \geq 1 - 2\exp{\left(-\tilde c_2 \log N |B_0| \right)}(1+o(1)).
	\end{align}

\section*{Acknowledgment}

This research has been partially supported by NSF grant DMS-1954343. The author wishes to express gratitude to his postdoctoral mentor Marek Biskup for helpful discussions and feedback, in addition to having introduced the problem to him in the first place and providing financial support. 

\setlength\parskip{0pt}

\bibliographystyle{alpha}
\bibliography{bib}

\end{document}